\documentclass[11pt]{amsart}

\usepackage{amsmath}
\usepackage{amssymb}
\usepackage{amsfonts}
\usepackage{mathrsfs}

\vfuzz2pt \DeclareMathOperator{\RE}{Re}
 
\DeclareMathOperator{\num1}{1}

\DeclareMathOperator{\id}{id}

 \DeclareMathOperator{\Ad}{Ad}
 \DeclareMathOperator{\sign}{sign}
\DeclareMathOperator{\TRR}{Tr} \DeclareMathOperator{\MOD}{mod}
 \DeclareMathOperator{\spec}{spec}

\DeclareMathOperator{\II}{i}

\newtheorem{theorem}{Theorem}

\newenvironment{teorem}[2][Theorem]{\begin{trivlist}
\item[\hskip \labelsep {\bfseries #1}\hskip \labelsep {\bfseries #2}]}{\end{trivlist}}

\newenvironment{cor}[2][Corollary]{\begin{trivlist}
\item[\hskip \labelsep {\bfseries #1}\hskip \labelsep {\bfseries #2}]}{\end{trivlist}}

\newenvironment{remm}[2][Remark]{\begin{trivlist}
\item[\hskip \labelsep {\bfseries #1}\hskip \labelsep {\bfseries #2}]}{\end{trivlist}}

\numberwithin{equation}{section}
\numberwithin{theorem}{section}

\newcommand{\set}[1]{\left\{#1\right\}}
\newcommand{\abs}[1]{\left\vert#1\right\vert}
\newcommand{\br}[1]{\left(#1\right)}
\newcommand{\SqBr}[1]{\left[#1\right]}

\begin{document}

\title[Order of zeta functions]{Order of zeta functions for compact odd-dimensional locally symmetric spaces}
\author{Muharem Avdispahi\'c and D\v zenan Gu\v si\'c}

\address{University of Sarajevo, Department of Mathematics, Zmaja od Bosne
35, 71000 Sarajevo, Bosnia and Herzegovina}
\email{\textbf{mavdispa@pmf.unsa.ba}}

\address{University of Sarajevo, Department of Mathematics, Zmaja od Bosne
35, 71000 Sarajevo, Bosnia and Herzegovina}
\email{\textbf{dzenang@pmf.unsa.ba}}



\maketitle

\begin{abstract}
We prove that the meromorphic continuations of the Ruelle and Selberg zeta functions considered by Bunke and Olbrich are of finite order not larger than the dimension of the underlaying compact, odd-dimensional, locally symmetric space.
\end{abstract}

\section{Introduction}

Let $Y=\Gamma\backslash G/K=\Gamma\backslash X$ be a compact, $n-$ dimensional ($n$ odd), locally
symmetric Riemannian manifold with negative sectional curvature, where $G$ is
a connected semisimple Lie group of real rank one, $K$ is a
maximal compact subgroup of $G$ and $\Gamma$ is a discrete
co-compact torsion free subgroup of $G$.

The covering manifold $X$ is known to be a real hyperbolic space $H\mathbb{R}^{n}$.

We require $G$ to be linear in order to have complexification
available.

In \cite{Bunke}, authors derived the properties of zeta functions canonically associated with the geodesic flow of $Y$. For even $n$, we proved that such functions are quotients of some entire functions whose order is not larger than $n$ (see, \cite{AG}). In this paper, which can be considered as a natural continuation of our previous study, we prove analogous results for odd $n$.

\section{Preliminaries}

Denote by $\mathfrak{g}$, $\mathfrak{k}$ the Lie algebras of $G$, $K$ and let $\mathfrak{g}=\mathfrak{k}\oplus\mathfrak{p}$ be the Cartan decomposition with Cartan involution $\theta$. Fix a one-dimensional subspace $\mathfrak{a}\subset\mathfrak{p}$ and let $M$ be the centralizer of $\mathfrak{a}$ in $K$ with Lie algebra $\mathfrak{m}$.

We denote by $(.,.)$ the $\Ad{\br{G}}-$invariant scalar product on $\mathfrak{g}$ which is normalized to restrict to the metric on $\mathfrak{p}$.

The unit sphere bundle $SX$ of $X$ may be represented as the homogeneous space $G/M$. Therefore, $SY=\Gamma\backslash G/M$.

Let $G=KAN$ and $\mathfrak{g}=\mathfrak{k}\oplus\mathfrak{a}\oplus\mathfrak{n}$ be the Iwasawa decompositions of $G$ and $\mathfrak{g}$, let $\Phi^{+}\br{\mathfrak{g},\mathfrak{a}}$ be the positive root system and $W$ its Weyl group. Put

\[\rho=\frac{1}{2}\sum\limits_{\alpha\in\Phi^{+}\br{\mathfrak{g},\mathfrak{a}}}m_{\alpha}\alpha,\]
\newline
where $m_{\alpha}$ is the dimension of the root space corresponding to $\alpha$.

Let $A^{+}=\exp\br{\mathfrak{a}^{+}}\subset A$, where $\mathfrak{a}^{+}$ is the positive Weyl chamber, i.e., the half line in $\mathfrak{a}$ on which the positive roots take positive values.

We choose a maximal abelian subalgebra $\mathfrak{t}$ of $\mathfrak{m}$. Now, $\mathfrak{h}=\mathfrak{t}_{\mathbb{C}}\oplus\mathfrak{a}_{\mathbb{C}}$ is a Cartan subalgebra of $\mathfrak{g}_{\mathbb{C}}$. We choose a positive root system $\Phi^{+}\br{\mathfrak{g}_{\mathbb{C}},\mathfrak{h}}$ having the property that
$\alpha_{|\mathfrak{a}}\in\Phi^{+}\br{\mathfrak{g},\mathfrak{a}}$
implies
$\alpha\in\Phi^{+}\br{\mathfrak{g}_{\mathbb{C}},\mathfrak{h}}$ for $\alpha\in\Phi\br{\mathfrak{g}_{\mathbb{C}},\mathfrak{h}}$.
Let
\[\delta=\frac{1}{2}\sum\limits_{\alpha\in\Phi^{+}\br{\mathfrak{g}_{\mathbb{C}},\mathfrak{h}}}\alpha\]
and set $\rho_{\mathfrak{m}}=\delta-\rho$.

The inclusion $M\subset K$ induces a restriction map $i^{*}:R\br{K}\rightarrow R\br{M}$, where $R\br{K}$, $R\br{M}$ are the representation rings with integer coefficients of the groups $K$, $M$, respectively.

Since $X=H\mathbb{R}^{n}$, we have $K=Spin\br{n}$, $M=Spin\br{n-1}$ or $K=SO\br{n}$, $M=SO\br{n-1}$ for $n\geq 3$.

Following \cite[p.~27]{Bunke}, we distinguish between two cases.

\textit{Case} (\textit{a}): $\sigma\in\hat{M}$ is invariant under the action of the Weyl group $W$.

Choose $\gamma\in R\br{K}$ such that $i^{*}\br{\gamma}=\sigma$. Note that $\gamma$ is uniquely determined by this condition. More explicitly, let

\begin{equation}\label{first}
\gamma=\sum a_{i}\gamma_{i},
\end{equation}
\newline
where $a_{i}\in\mathbb{Z}$ and $\gamma_{i}\in\hat{K}$. Form

\[V_{\gamma}^{\pm}=\sum\limits_{\sign{\br{a_{i}}}=\pm1}\sum\limits_{m=1}^{\abs{a_{i}}}V_{\gamma_{i}},\]
\newline
where $V_{\gamma_{i}}$ is the representation space of
$\gamma_{i}$. Then $V\br{\gamma}^{\pm}$ is given by $V\br{\gamma}^{\pm}=G\times_{K}V_{\gamma}^{\pm}$. Now, $V\br{\gamma}=V\br{\gamma}^{+}\oplus V\br{\gamma}^{-}$ is a $\mathbb{Z}_{2}$-graded homogeneous vector bundle on $X$ associated to $\gamma$ and

\[V_{Y,\chi}\br{\gamma}=\Gamma\backslash\br{V_{\chi}\otimes
V\br{\gamma}}\]
\newline
is a $\mathbb{Z}_{2}$-graded vector bundle on $Y$, where
$\br{\chi,V_{\chi}}$ is a finite-dimensional unitary
representation of $\Gamma$.

Using the scalar product $(.,.)$, we define the shift constant

\[c\br{\sigma}=\abs{\rho}^{2}+\abs{\rho_{\mathfrak{m}}}^{2}-\abs{\mu_{\sigma}+\rho_{\mathfrak{m}}}^{2},\]
\newline
where $\mu_{\sigma}\in$i$\mathfrak{t}^{*}$ is the highest weight of
$\sigma$ (see, \cite[pp. 19--20]{Bunke}). The aforementioned scalar product also fixes the Casimir operator $\Omega$ acting on sections of the bundle $V\br{\gamma}$. Finally, we define the operator (see, \cite[p.~28]{Bunke})

\[A_{Y,\chi}\br{\gamma,\sigma}^{2}=-\Omega -c\br{\sigma}: C^{\infty}\br{Y,V_{Y,\chi}\br{\gamma}}\rightarrow C^{\infty}\br{Y,V_{Y,\chi}\br{\gamma}}.\]
\newline
\textit{Case} (\textit{b}): $\sigma\in\hat{M}$ is not invariant under the action of the Weyl group.

Choose $\gamma^{'}\in\hat{Spin\br{n}}$ satisfying $s\otimes\gamma^{'}=\gamma^{+}\oplus\gamma^{-}$, where $s$ is the spin representation of $Spin\br{n}$ and $\gamma^{\pm}$ are representations of $K$ such that for the nontrivial element $w\in W$ we have

\[\sigma-w\sigma=\sign\br{\nu_{k}}\br{s^{+}-s^{-}}i^{*}\br{\gamma^{'}},\]

\[\sigma+w\sigma=i^{*}\br{\gamma^{+}-\gamma^{-}},\]
\newline
where $\nu_{k}$ is the last coordinate of the highest weight of $\sigma$ and $s^{\pm}$ are the half-spin representations of $Spin\br{n-1}$.

Notice that $\gamma^{'}$ is unique.

Set $\gamma=\gamma^{+}-\gamma^{-}\in R\br{K}$ and $\gamma^{s}=\gamma^{+}+\gamma^{-}\in R\br{K}$. We define the bundles
\[V\br{\gamma},\quad V_{Y,\chi}\br{\gamma},\quad V\br{\gamma^{s}},\quad V_{Y,\chi}\br{\gamma^{s}}\]
\newline
and the operators
\[A_{Y,\chi}\br{\gamma,\sigma},\quad A_{Y,\chi}\br{\gamma^{s},\sigma}\]
\newline
in the same way as in the \textit{case} (\textit{a}).

Since $V_{Y,\chi}\br{\gamma^{s}}$ is a Clifford bundle, it carries a Dirac operator $D_{Y,\chi}\br{\sigma}$. We make this Dirac operator unique reasoning exactly as in \cite[p.~29]{Bunke}.

Being selfadjoint, the Dirac operator $D_{Y,\chi}\br{\sigma}$ satisfies

\[A_{Y,\chi}\br{\gamma^{s},\sigma}=\abs{D_{Y,\chi}\br{\sigma}}.\]
\newline
Let $E_{A}\br{.}$ be the family of spectral projections of
a normal operator $A$. For $s\in\mathbb{C}$, we define
\[m_{\chi}\br{s,\gamma,\sigma}=\TRR{E_{A_{Y,\chi}\br{\gamma,\sigma}}}\br{\set{s}},\]

\[m_{\chi}^{s}\br{s,\sigma}=\TRR{\br{E_{D_{Y,\chi}\br{\sigma}}\br{\set{s}}-E_{D_{Y,\chi}\br{\sigma}}\br{\set{-s}}}}.\]
\newline
Note that the multiplicities $m_{\chi}\br{s,\gamma,\sigma}$, $m_{\chi}^{s}\br{s,\sigma}$ do not depend on the choice of the representative (\ref{first}) of $\gamma$.

Define the root vector $H_{\alpha}\in\mathfrak{a}$ for
$\alpha\in\Phi^{+}\br{\mathfrak{g},\mathfrak{a}}$ by

\[\lambda\br{H_{\alpha}}=\frac{\br{\lambda,\alpha}}{\br{\alpha,\alpha}},\quad\forall\lambda\in\mathfrak{a}^{*}.\]
\newline
Requiring that
\[e^{2\pi\II\varepsilon_{\alpha}\br{\sigma}}=\sigma\br{e^{2\pi
\II H_{\alpha}}}\in\set{\pm 1}\]
\newline
for $\alpha\in\Phi^{+}\br{\mathfrak{g},\mathfrak{a}}$, we define
$\varepsilon_{\alpha}\br{\sigma}\in\set{0,\frac{1}{2}}$.

Recall that $\Phi^{+}\br{\mathfrak{g},\mathfrak{a}}=\set{\alpha}$ or
$\Phi^{+}\br{\mathfrak{g},\mathfrak{a}}=\set{\frac{\alpha}{2},\alpha}$, where $\alpha$ is the long root. For such $\alpha$ and $\sigma\in\hat{M}$, we introduce $\epsilon_{\sigma}\in\set{0,\frac{1}{2}}$ by

\[\epsilon_{\sigma}\equiv\frac{\abs{\rho}}{T}+\varepsilon_{\alpha}\br{\sigma}\,\MOD{\mathbb{Z}},\]
\newline
the lattice $L\br{\sigma}\subset\mathbb{R}\cong\mathfrak{a}^{*}$ by
$L\br{\sigma}=T\br{\epsilon_{\sigma}+\mathbb{Z}}$ and $P_{\sigma}\br{\lambda}$ by

\[P_{\sigma}\br{\lambda}=\prod\limits_{\beta\in\Phi^{+}\br{\mathfrak{g}_{\mathbb{C}},\mathfrak{h}}}
\frac{\br{\lambda
+\mu_{\sigma}+\rho_{\mathfrak{m}},\beta}}{\br{\delta,\beta}},\quad\lambda\in\mathfrak{a}_{\mathbb{C}}^{*}\cong\mathbb{C},\]
\newline
where $T=\abs{\alpha}$.

\section{Zeta functions}

We denote by $C\Gamma$ the set of conjugacy classes of $\Gamma$.

$\Gamma$ being co-compact and torsion free, each $\gamma\in\Gamma$, $\gamma\neq\num1$ is hyperbolic.

It is well known that for every hyperbolic $g\in G$ there is an Iwasawa decomposition $G=NAK$ such that $g=am\in A^{+}M$.

Let $\varphi$ be a canonical dynamical system on $SY$. It is given  by

\[\varphi: \mathbb{R}\times SY\ni\br{t,\Gamma gM}\rightarrow\Gamma g\exp\br{-tH}M\in SY,\]
\newline
where $H\in\mathfrak{a}^{+}$ is the unit vector. Define

\[\varphi_{\chi,\sigma}: \mathbb{R}\times V_{\chi}\br{\sigma}\ni\br{t,\SqBr{g,v\otimes w}}\rightarrow
\SqBr{g\exp\br{-tH},v\otimes w}\in V_{\chi}\br{\sigma},\]
\newline
where $V_{\chi}\br{\sigma}=\Gamma\backslash\br{G\times_{M}V_{\sigma}\otimes
V_{\chi}}$ denotes the vector bundle associated to finite-dimensional unitary representations $\br{\sigma,V_{\sigma}}$, $\br{\chi,V_{\chi}}$ of $M$,
$\Gamma$, respectively. Then, $\varphi_{\chi,\sigma}$ is a lift of $\varphi$ to $V_{\chi}\br{\sigma}$.

As easily seen, the free homotopy classes of closed paths on $Y$ are in a natural one-to-one correspondence with the set $C\Gamma$.

If $\SqBr{1}\neq\SqBr{g}\in C\Gamma$, then the corresponding closed orbit is

\[c=\set{\Gamma g^{'}\exp\br{-tH}M\,\,|\,\, t\in\mathbb{R}},\]
\newline
where $g^{'}$ is choosen so that $\br{g^{'}}^{-1}gg^{'}=ma\in MA^{+}$, while its length $l\br{c}$ is given by $\abs{\log\br{a}}$.

The lift of $c$ to $V_{\chi}\br{\sigma}$ induces a linear transformation $\mu_{\chi,\sigma}\br{c}$ on the fibre over $\Gamma g^{'}M$, the monodromy of $c$ (see, \cite[p.~96]{Bunke}).

For $s\in\mathbb{C}$, $\RE{\br{s}}>2\rho$, the Ruelle zeta function is defined by

\[Z_{R,\chi}\br{s,\sigma}=\prod\limits_{c\,\,\textrm{prime}}\det\br{1-\mu_{\chi,\sigma}\br{c}e^{-sl\br{c}}}^{\br{-1}^{n-1}}.\]
\newline
Here, a closed orbit $c$ through $y\in SY$ corresponding to $\SqBr{g}\in C\Gamma$ is called prime if $\abs{\log\br{a}}$ is the smallest time with $\varphi\br{\abs{\log\br{a}},y}=y$.

We point out that $Z_{R,\chi}\br{s,\sigma}$ is associated to the flow $\varphi_{\chi,\sigma}$.

The geodesic flow $\varphi$ satisfies the Anosov property, i.e., there is a $d\varphi$-invariant splitting
\[TSY=T^{s}SY\oplus T^{0}SY\oplus T^{u}SY,\]
\newline
where $T^{0}SY$ consists of vectors tangential to the orbits, while the vectors in $T^{s}SY$ ($T^{u}SY$) shrink (grow) exponentially with respect to the metric as $t\rightarrow\infty$, when transported with $d\varphi$. In our case, the splitting is given by

\[TSY\cong\Gamma\backslash G\times_{M}\br{\bar{\mathfrak{n}}\oplus\mathfrak{a}\oplus\mathfrak{n}},\]
\newline
where $\bar{\mathfrak{n}}=\theta\mathfrak{n}$. Now, the monodromy $P_{c}$ in $TSY$ of a closed orbit $c$, decomposes according to the splitting

\[P_{c}=P_{c}^{s}\oplus\id\oplus P_{c}^{u}.\]
\newline
For $s\in\mathbb{C}$, $\RE{\br{s}}>\rho$, the Selberg zeta function is the infinite product

\[Z_{S,\chi}\br{s,\sigma}=\prod\limits_{c\,\,\textrm{prime}}\prod\limits_{k=0}^{\infty}\det\br{1-\mu_{\chi,\sigma}\br{c}\otimes S^{k}\br{P_{c}^{s}}e^{-\br{s+\rho}l\br{c}}},\]
\newline
where $S^{k}$ is the $k$-th symmetric power of an endomorphism.

In \textit{case} (\textit{b}), the authors \cite[pp. 97--98]{Bunke} also defined

\[S_{\chi}\br{s,\sigma}=Z_{S,\chi}\br{s,\sigma}Z_{S,\chi}\br{s,w\sigma}\]
\newline
and the super zeta function

\[S_{\chi}^{s}\br{s,\sigma}=\frac{Z_{S,\chi}\br{s,\sigma}}{Z_{S,\chi}\br{s,w\sigma}}\]
\newline
for the non-trivial element $w\in W$.

It is known \cite{Fried}, that the Ruelle zeta function can be expressed in terms of Selberg zeta functions.

We have
\begin{equation}\label{second}
Z_{R,\chi}\br{s,\sigma}=\prod\limits_{p=0}^{n-1}\prod\limits_{\br{\tau,\lambda}\in I_{p}}Z_{S,\chi}\br{s+\rho-\lambda,\tau\otimes\sigma}^{\br{-1}^{p}},
\end{equation}
\newline
where
\[I_{p}=\set{\br{\tau,\lambda}\mid\tau\in\hat{M},\lambda\in\mathbb{R}}\]
\newline
are such sets that $\Lambda^{p}\mathfrak{n}_{\mathbb{C}}$ decomposes with respect to $MA$ as

\[\Lambda^{p}\mathfrak{n}_{\mathbb{C}}=\sum\limits_{\br{\tau,\lambda}\in I_{p}}V_{\tau}\otimes\mathbb{C}_{\lambda}.\]
\newline
Here, $V_{\tau}$ is the space of the representation $\tau$ and $\mathbb{C}_{\lambda}$, $\lambda\in\mathbb{C}$ is the one-dimensional representation of $A$ given by  $A\ni a\rightarrow a^{\lambda}$.

The following theorem holds true (see, \cite[p.~113, Th. 3.15]{Bunke}).
\newline

\begin{teorem}{A.}\label{t3.15}
\textit{Zeta functions $Z_{S,\chi}\br{s,\sigma}$, $S_{\chi}\br{s,\sigma}$ and $S_{\chi}^{s}\br{s,\sigma}$ have
meromorphic continuation to all of $\mathbb{C}$. The singularities of $Z_{S,\chi}\br{s,\sigma}$} (\textit{case} (\textit{a})) \textit{and of $S_{\chi}\br{s,\sigma}$} (\textit{case} (\textit{b})) \textit{are}\\
\begin{itemize}
    \item[] at $\pm\II$$s$ of order $m_{\chi}\br{s,\gamma,\sigma}$ if
    $s\neq 0$ is an eigenvalue of
    $A_{Y,\chi}\br{\gamma,\sigma}$,\\
    \item[] at $s=0$ of order $2m_{\chi}\br{0,\gamma,\sigma}$ if
    $0$ is an eigenvalue of $A_{Y,\chi}\br{\gamma,\sigma}$.\\
\end{itemize}
\textit{In} \textit{case} (\textit{b}), \textit{the singularities of $S_{\chi}^{s}\br{s,\sigma}$ are at $\II$$s$ and have the order $m_{\chi}^{s}\br{s,\sigma}$ if $s\in\mathbb{R}$ is an eigenvalue of $D_{Y,\chi}\br{\sigma}$. Furthermore, in} \textit{case} (\textit{b}), \textit{the zeta function $Z_{S,\chi}\br{s,\sigma}$ has singularities at $\II$$s$, $\pm s\in\spec\br{A_{Y,\chi}\br{\gamma^{s},\sigma}}$ of order\\ $\frac{1}{2}\br{m_{\chi}\br{\abs{s},\gamma,\sigma}+m_{\chi}^{s}\br{s,\sigma}}$ if $s\neq 0$ and $m_{\chi}\br{0,\gamma,\sigma}$ if $s=0$.}
\end{teorem}\

\section{Main result}

The main result of this paper is the following theorem
\newline

\begin{theorem}\label{new}
If $f\br{s}\in\set{Z_{S,\chi}\br{s,\sigma}, S_{\chi}\br{s,\sigma}, S_{\chi}^{s}\br{s,\sigma}}$, then there exist entire functions $Z_{1}\br{s}$, $Z_{2}\br{s}$ of order at most $n$ such that

\[f(s)=\frac{Z_{1}\br{s}}{Z_{2}\br{s}},\]
\newline
where the zeros of $Z_{1}\br{s}$ correspond to the zeros of $f\br{s}$ and the zeros of $Z_{2}\br{s}$ correspond to the poles of $f\br{s}$. The orders of the zeros of $Z_{1}\br{s}$ resp. $Z_{2}\br{s}$ equal the orders of the corresponding zeros. resp. poles of $f\br{s}$.
\end{theorem}
\begin{proof}
Let $N\br{r}=\#\set{s\in\spec{D_{Y,\chi}\br{\sigma}}\,\,|\,\,\abs{s}\leq r}$.

Since $D_{Y,\chi}^{2}\br{\sigma}=A_{Y,\chi}^{2}\br{\gamma^{s},\sigma}$ and $A_{Y,\chi}^{2}\br{\gamma^{s},\sigma}$ is an elliptic operator of the second order, we have the estimate (see, \cite[p.~21]{Bunke2})

\[N\br{r}\sim Cr^{n},\]
as $r\rightarrow +\infty$.

Denote by $S$ the set of singularities of $f\br{s}$.

If $f\br{s}=Z_{S,\chi}\br{s,\sigma}$ (\textit{case} (\textit{a})) or $f\br{s}=S_{\chi}\br{s,\sigma}$ (\textit{case} (\textit{b})), then (see, \cite[p.~529, Eq. (4.7)]{AG})
\begin{equation}\label{third}
\sum\limits_{s\in S\backslash\set{0}}\abs{s}^{-\br{n+\varepsilon}}=O\br{1},
\end{equation}
for any $\varepsilon>0$.

Let $R_{1}$, $R_{2}$ denote respectively the sets consisting of the zeros, poles of $f\br{s}$. For simplicity, the point $s=0$ will be considered separately. Assume that $0\notin R_{i}$, $i=1,2$. Put $m_{0}=2m_{\chi}\br{0,\gamma,\sigma}$.

It follows from (\ref{third}) that
\begin{equation}\label{fourth}
\sum\limits_{s\in R_{i}}\abs{s}^{-\br{n+\varepsilon}}<\infty,
\end{equation}
for $i=1,2$ and for any $\varepsilon>0$.

Let $\rho_{1}^{i}$ resp. $p_{i}$ denote the convergence exponent resp. the genus of the set $R_{i}$ for $i=1,2$ (see, \cite[p.~14]{Boas}). By (\ref{fourth}), $\rho_{1}^{i}$, $p_{i}\leq n$ for $i=1,2$. Now, by \cite[p.~19, Th. 2.6.5.]{Boas}, $W_{i}\br{s}$ is an entire function of order $\rho_{1}^{i}$ over $\mathbb{C}$, where

\[W_{i}\br{s}=\prod\limits_{z\in R_{i}}E\br{\frac{s}{z},p_{i}},\]

\[E\br{u,k}=\br{1-u}\exp\br{u+\frac{u^{2}}{2}+...+\frac{u^{k}}{k}},\]
$i=1,2$.

We see that $f\br{s}W_{1}\br{s}^{-1}W_{2}\br{s}s^{-m_{0}}$ is an entire function and has no zeros over $\mathbb{C}$. Hence, (see, e.g., \cite{Conway}), there exists an analytic function $g\br{s}$ such that

\[f\br{s}W_{1}\br{s}^{-1}W_{2}\br{s}s^{-m_{0}}=e^{g\br{s}}\]
\newline
for $s\in\mathbb{C}$. By taking logarithms of both sides, we obtain

\[g\br{s}=\log f\br{s}+\log W_{2}\br{s}-\log W_{1}\br{s}-m_{0}\log s.\]
\newline
Differentiating $n+1$ times and having in mind that the logarithmic derivative of $f\br{s}$ is given by a Dirichlet series absolutely convergent for $\RE\br{s}\gg 0$ (see, \cite{Bunke}), we conclude that
\[\lim\limits_{\abs{s}\rightarrow +\infty}g^{\br{n+1}}\br{s}=0.\]
\newline
Therefore, the degree of $g\br{s}$ is at most $n$. Now, the assertion follows from the representation
\[f\br{s}=s^{m_{0}}e^{g\br{s}}\frac{W_{1}\br{s}}{W_{2}\br{s}}.\]
\newline
If $f\br{s}=S_{\chi}^{s}\br{s,\sigma}$  (\textit{case} (\textit{b})), then

\[\sum\limits_{s\in S\backslash\set{0}}\abs{s}^{-\br{n+\varepsilon}}=\sum_{\substack{s\in
S\backslash\set{0}\\
0<\abs{s}<1}}\abs{s}^{-\br{n+\varepsilon}}+\sum_{\substack{s\in
S\backslash\set{0}\\
\abs{s}\geq 1}}\abs{s}^{-\br{n+\varepsilon}}=\]

\[O\br{1}+\sum_{\substack{s\in
\spec{D_{Y,\chi}\br{\sigma}}\\
\abs{s}\geq 1}}\abs{m_{\chi}^{s}\br{s,\sigma}}\abs{s}^{-\br{n+\varepsilon}}=\]

\[O\br{\int\limits_{1}^{+\infty}t^{-\br{n+\varepsilon}}dN\br{t}}=O\br{1},\]
\newline
for any $\varepsilon>0$. By the same argumentation as in the previous case, the assertion follows.

Finally, if $f\br{s}=Z_{S,\chi}\br{s,\sigma}$ (\textit{case} (\textit{b})), the theorem follows from the fact that $Z_{S,\chi}\br{s,\sigma}=\sqrt{S_{\chi}\br{s,\sigma} S_{\chi}^{s}\br{s,\sigma}}$. This completes the proof.
\end{proof}\

\begin{cor}{4.2}\label{crll}
\textit{A meromorphic extension over $\mathbb{C}$ of the Ruelle zeta function\\ $Z_{R,\chi}\br{s,\sigma}$ can be expressed as}

\[Z_{R,\chi}\br{s,\sigma}=\frac{Z_{1}\br{s}}{Z_{2}\br{s}},\]
\newline
\textit{where $Z_{1}\br{s}$, $Z_{2}\br{s}$ are entire functions of order at most $n$ over $\mathbb{C}$.}
\end{cor}
\begin{proof}
An immediate consequence of the formula (\ref{second}) and Theorem \ref{new}.
\end{proof}\

\begin{remm}{4.3.}
An approach using the Selberg zeta function is not always sufficient to reach expected error terms in the prime geodesic theorem. As explained in \cite{Park}, this approach provide us only with the error terms corresponding to the poles of the logarithmic derivative of the Selberg zeta lying in the strip $n-2<\RE\br{s}\leq n-1$. In case $n>3$, the meromorphic continuation of the Ruelle zeta function yields more satisfactory results (see, \cite{AG0}, \cite{Park}, \cite{Randol1}).
\end{remm}

\end{document}